\documentclass{amsart}
\usepackage[left=2.5cm,    
right=2.5cm,   
top=2cm,       
bottom=2cm]{geometry} 
\usepackage{amssymb}
\usepackage{amsmath}
\usepackage{mathtools}
\usepackage{verbatim}

\usepackage{bm} 
\usepackage[colorlinks,linkcolor=cyan,citecolor=cyan]{hyperref}
\usepackage[hang,flushmargin]{footmisc} 
\usepackage[square,comma,sort&compress,numbers]{natbib} 
\usepackage{mathrsfs} 
\usepackage[font=footnotesize,skip=0pt,textfont=rm,labelfont=rm]{caption,subcaption} 

\usepackage{amsthm}
\newtheorem{theorem}{Theorem}[section]
\newtheorem{proposition}{Proposition}[section]
\newtheorem{lemma}{Lemma}[section]
\newtheorem{definition}{Definition}[section]

\newtheorem{corollary}{Corollary}[section]
\newtheorem{remark}{Remark}[section]

\begin{document}
	\title{Nonexistence of weakly stable Yang-Mills fields}
	\author{Xiaoli Han}
	\address{Xiaoli Han\\ Math department of Tsinghua university\\ Beijing\\ 100084\\ China\\} \email{hanxiaoli@mail.tsinghua.edu.cn}
	\author{Yang Wen}
	\address{Yang Wen\\ Beijing International Center for Mathematical Research\\ Peking University\\ Beijing\\ 100871\\ China\\} \email{2406397066@pku.edu.cn}
	\begin{abstract}
		In this paper we prove that there is a neighborhood in the  $C^2$ topology of the usual metric on the
		Euclidean sphere $S^n (n\geq 5)$ such that there is no nontrivial weakly stable Yang-Mills connections for any metric $\tilde{g}$ in this neighborhood. We also study the stability of Yang-Mills connections on the warped product manifolds.
	\end{abstract}
	\maketitle
	\section{Introduction}\  
	Let $(M,g)$ be an $n$-dimensional Riemannian manifold and $E\to M$ a vector bundle of rank $r$ endowed with a structure group $G$ and compatible Riemannian metric $\langle\ ,\ \rangle$. The Yang-Mills functional on the affine space of metric-compatible connections of $E$ is defined as
	\begin{align*}
		\operatorname{YM}(\nabla)=\frac12\int_M|R^\nabla|_g^2dV_g,
	\end{align*}
	where $\nabla$ is a connection of $E$ preserving $\langle\ ,\ \rangle$ and $R^\nabla$ denotes its curvature form. Let $d^\nabla$ be the exterior covariant derivative induced by $\nabla$, and $\delta^\nabla$ its formal $L^2$-adjoint. The associated Euler-Lagrange equation takes the form
	\begin{align}\label{YM equation}
		\delta^\nabla R^\nabla=0,
	\end{align}
	with solutions termed Yang-Mills connections. 
	
	Our work focuses on locally minimizing Yang-Mills connections – critical points of $\operatorname{YM}$ achieving local minima. Such a connection $\nabla$ is called weakly stable if the second variation of $\operatorname{YM}$ is non-negative at $\nabla$:
	\begin{align*}
		\frac{d^2}{dt^2}\operatorname{YM}(\nabla^t)\mid_{t=0}\ge0
	\end{align*}
	for any smooth curve of connections $\nabla^t$ with $\nabla^0=\nabla$.
	
	The stability analysis of Yang-Mills functionals originates in Xin's seminal work \cite{Xin}, which established that no non-constant stable harmonic map exist from $S^n (n\ge3)$ to any Riemannian manifold. Adapting this methodology, Bourguignon and Lawson \cite{BL} proved that there is nontrivial weakly stable Yang-Mills connections on $S^n$ for $n\ge5$. They also proved that if the structure group $G=SU(2), SU(3)$, or $U(2)$, then the weakly stable Yang-Mills connection on $S^4$ is self-dual or anti self-dual. Laquer \cite{Laquer} extended this program by classifying symmetric manifolds admitting unstable Yang-Mills fields. Through explicit computation of the Yang-Mills functional's Morse index and nullity, it was shown that compact symmetric spaces admit weakly stable Yang-Mills connections except for $S^n (n\ge5)$, compact simple Lie groups, $SP(p+1)/SP(p)\times SP(1)$, $E_6/F_4$ and $F_4/Spin(9)$. Further developments by Kobayashi, Ohnita, and Takeuchi \cite{SYM} demonstrated that all Yang-Mills connections on $P^2(\mathbb{C}ay)$ and $E_6/F_4$ is unstable. Most recently, Ni \cite{Ni} proved that the Morse index of $\operatorname{YM}$ on $S^n (n\ge5)$ is larger than $m+1$, and the smallest eigenvalue of the Jacobi operator of $\operatorname{YM}$ is not greater than $4-n$.
	
	Lawson and Simons in \cite{LS} proved that there are no compact stable minimal submanifolds in the Euclidean spheres. They conjectured that there are no stable minimal submanifolds in any compact, simply connected, Riemannian manifold $M$ that is $1/4$-pinched. Here, by $\delta$-pinched for some $\delta>0$, we mean that at each point of $M$, the sectional curvatures are positive and the ratio between the smallest and the largest sectional curvatures of $M$, at that point, is strictly bigger than $\delta$. This conjecture has been verified for several classes of manifolds. Lawson and Simons \cite{LS} show it holds if $M^n$ can be isometrically immersed in a standard sphere with sufficiently small second fundamental form.  Howard \cite{H} showed that for each $n \geqq 3$, there exists a constant $\delta(n,p)$ satisfying $1 / 4<\delta(n,p)<1$ such that if $M^n$ is a simply-connected compact strictly $\delta(n,p)$-pinched Riemannian manifold of dimension $n$, then there are no nonconstant $p-$dimensional stable minimal submanifolds. But unfortunately $\lim _{n \rightarrow \infty} \delta(n)=1$. Later, Shen and He \cite{SH} proved that there are no compact stable minimal submanifolds in a compact simply-connected $0.77$-pinched Riemannian manifold. 
	
	Naturally, there is a "harmonic version" of this famous conjecture on stable minimal submanifolds. Howard \cite{H}  proved there exists a constant $1/4<\delta(n)<1$ such that if $M^n$ is a simply-connected compact strictly $\delta(n)$-pinched Riemannian manifold of dimension $n$, then there are no nonconstant stable harmonic maps $\psi: N\to M$ for any compact manifold $N$.  Okayasu gave a dimension-independent pinching constant $\delta=0.83$ in \cite{Ok}. 
	
	Similarly, we can consider "Yang-Mills version" of this conjecture.
	
	{\it Question:} Let $M$ be compact, simply connected Riemannian manifold that is $1/4$-pinched. Is every stable Yang-Mills connection on $M$  flat ?
	
	Ohnita and Pan \cite{OP} confirmed this question when $M$ is $\delta(n)-$ pinched. But in their argument it is not possible to find a pinching constant $\delta$  independent of the
	dimension of the base manifold $M$. We aim to address this problem in the conformal class of $S^n$.Conformal metrics are indispensable in geometric PDE theory, where their scale invariance simplifies blow-up analysis and resolves singularities through controlled rescalings. Han \cite{Han} excluded stable quasi-harmonic spheres in Gaussian conformal spaces. Cherif-Djaa-Zegga \cite{AMK} proved that under $\operatorname{Hess}(e^{(n-2)\varphi})\le0$, any stable harmonic map from $(S^n,e^{2\varphi}g_{standard})$ to a Riemannian manifold is a constant map. Franz-Trinca \cite{FT} established non-existence of closed stable minimal $k$-submanifolds in generic conformal spheres $(S^n,\tilde g)$ for dimensions $2\le k\le n-\delta^{-1}$, where $\delta$ dependents on $\tilde g$.
	
	The purpose of the present work is to study the existence of weakly stable Yang-Mills connections on conformal spheres. We prove the following theorem.
	\begin{theorem}\label{thm 1}
		Assume $(S^n,g)$ is the standard sphere and $\tilde g=e^{2\varphi}g$ is a conformal metric. If $n\ge5$ and
		\begin{align*}
			\frac12\Delta\varphi-\frac{n-4}2|\operatorname{grad}(\varphi)|_g^2+2>0,
		\end{align*}
		then there is no nontrivial weakly stable Yang-Mills connections on $(S^n,\tilde g)$.
	\end{theorem}
	
	Suppose \(\varphi\) is a small perturbation of a constant, then $\varphi$ satisfies the condition in the above theorem.
	\begin{corollary}
		There is a neighborhood in the  $C^2$ topology of the usual metric on the
		Euclidean sphere $S^n (n\geq 5)$ such that there is no nontrivial weakly stable Yang-Mills connections for any metric $\tilde{g}$ in this neighborhood.
	\end{corollary}
	Since the Yang-Mills functional is invariant under gauge transformation, it is meaningless to define a connection is stable if the second variation is positive. However, we define a connection as stable when the second variation is positive in all directions orthogonal to gauge transformations (see Definition \ref{def:stable} or \cite{BL} (6.12)). We obtain the following result.
	\begin{theorem}\label{thm 2}
		Assume $(S^n,g)$ is the standard sphere and $\tilde g=e^{2\varphi}g$ is a conformal metric. If
		\begin{equation*}
			(n-4)\left[\frac{1}2\Delta\varphi-\frac{n+4}2|\operatorname{grad}(\varphi)|_g^2+2\right]\ge0,
		\end{equation*}
		then there is no nontrivial stable Yang-Mills connections on $(S^n,\tilde g)$. Especially, there is no nontrivial stable Yang-Mills connections on $(S^4,\tilde g)$.
	\end{theorem}
	We further analyze Yang-Mills connections on a finite product of spheres, and derived the following result.
	\begin{theorem}\label{thm:S^n times S^m}
		Assume $M=S^{n_1}\times...\times S^{n_q}$ is a product space of finite standard spheres with the product metric $g_M=g_{S^{n_1}}\oplus...\oplus g_{S^{n_q}}$.\\
		(1) If $n_i\ge5\ (1\le i\le q)$, then there is no nontrivial weakly stable Yang-Mills connections on $M$.\\
		(2) If $n_i\ge4\ (1\le i\le q)$, then there is no nontrivial stable Yang-Mills connections on $M$.
	\end{theorem}
	Furthermore, we give the stability of Yang-Mills connections on some warped product manifolds.
	\begin{theorem}\label{thm 3}
		Assume $M=I\times N$ is a non-compact manifold with warped product metric $g_M=drdr+f^2(r)g_N$, where $(N,g_N)$ is an $n-1$ dimensional compact Riemannian manifold, $I\subset\mathbb{R}$ is an open integral, $f\in C^2(I,(0,+\infty))$ satisfies
		\begin{equation}\label{condition for f}
			\begin{split}
				&(a)\ f>0,\\
				&(b)\ f(r)f''(r)\in L^\infty(I),\\
				&(c)\ f(r)(f'(r)+1)=O(r)\textrm{ for }r\in I\textrm{ and as }r\to\infty,\\
				&(d)\ f(r)(f'(r)+1)=O(|r-a|)\textrm{ for }r\to a\in\partial I.
			\end{split}
		\end{equation}
		If $(n-4)f''<0$ and $\nabla$ is a weakly stable Yang-Mills connection on $M$ with $R^\nabla\in L^\infty(M)\cap L^2(M)$, then we have
		\begin{align*}
			i_{\frac\partial{\partial r}}R^\nabla=0.
		\end{align*}
	\end{theorem}
	As a corollary, we have the following proposition.
	\begin{corollary}
		Assume $M$ is an ellipsoid defined by $\{(x_1,...,x_{n+1})\in\mathbb{R}^{n+1}\mid\sum_{i=1}^{n+1}a_ix_i^2=1\}$ for some positive constant $a_i$. If $n\ge5$, then there is no nontrivial weakly stable Yang-Mills connections on $M$. Moreover, there is no nontrivial stable Yang-Mills connections on $M$ if $n\ge4$.
	\end{corollary}
	\section{Preliminary}
	\subsection{The connections and curvatures on vector bundles}\ 
	
	Assume $(M,g)$ be a $n$ dimensional compact Riemannian manifold and $\mathscr{X}(M)$ be the collection of vector fields on $M$. Let $E\to M$ be a rank $r$ vector bundle over $M$ with a compact Lie group $G$ as its structure group. We also assume $\langle\ ,\ \rangle$ be the Riemannian metric on $E$ compatible with the action of $G$ and $\mathfrak{g}_E$ be the adjoint bundle of $E$. Let $\nabla:\Omega^0(\mathfrak{g}_E)\to\Omega^1(\mathfrak{g}_E)$ be the connection on $E$ compatible with the metric $\langle\ ,\ \rangle$. Locally, $\nabla$ takes the form
	\begin{align*}
		\nabla=d+A,
	\end{align*}
	where $A\in\Omega^1(\mathfrak{g}_E)$.
	
	For any connection $\nabla$ on $E$, the corresponding curvature $R^\nabla$ is given by
	\begin{align*}
		R^\nabla=dA+\frac12[A\wedge A],
	\end{align*}
	where
	\begin{align*}
		\frac12[A\wedge A](X,Y)=[A(X),A(Y)].
	\end{align*}
	
	The induced inner product on $\Omega^p(\mathfrak{g}_E)$ is given by
	\begin{align*}
		\langle\phi,\psi\rangle=\frac12Tr(\phi^T\psi)
	\end{align*}
	for any $\phi,\psi\in\Omega^0(\mathfrak{g}_E)$ and
	\begin{align*}
		\langle\phi,\psi\rangle_g=\frac1{p!}\sum_{1\le i_1,...,i_p\le p}\langle\phi(e_{i_1},...,e_{i_p}),\psi(e_{i_1},...,e_{i_p})\rangle,
	\end{align*}
	for any $\phi,\psi\in\Omega^p(\mathfrak{g}_E)$, where $\{e_i\mid1\le i\le n\}$ is an orthogonal basis of $TS^n$ respect to the metric $g$. After integrating, we get the global inner product of $\Omega^p(\mathfrak{g}_E)$, that is
	\begin{align*}
		(\phi,\psi)_g=\int_{S^n}\langle\phi,\psi\rangle_g dV_g.
	\end{align*}
	The connection $\nabla$ induces a connection $d^\nabla:\Omega^p(\mathfrak{g}_E)\to\Omega^{p+1}(\mathfrak{g}_E)$ on $\Omega^p(\mathfrak{g}_E)$, and we assume $\delta^\nabla:\Omega^p(\mathfrak{g}_E)\to\Omega^{p-1}(\mathfrak{g}_E)$ be the formal adjoint operator of $d^\nabla$. On local coordinates, we have
	\begin{align*}
		d^\nabla\phi(X_1,...,X_{p+1})&=\sum_{i=1}^{p+1}(-1)^{i+1}\nabla_{X_i}\phi(X_1,...,\hat{X_i},...,X_{p+1}),\\
		\delta^\nabla\phi(X_1,...,X_{p-1})&=-\sum_{i=1}^n\nabla_{e_i}\phi(e_i,X_1,...,X_{p-1})
	\end{align*}
	for any $\phi\in\Omega^p(\mathfrak{g}_E)$. 
	
	We can define the Laplace–Beltrami operator $\Delta^\nabla$ by
	\begin{align*}
		\Delta^\nabla\phi=d^\nabla\delta^\nabla\phi+\delta^\nabla d^\nabla\phi
	\end{align*}
	and the rough Laplacian operator $\nabla^\ast\nabla$ by
	\begin{align*}
		\nabla^\ast\nabla\phi=-\sum_{i=1}^n(\nabla_{e_i}\nabla_{e_i}\phi-\nabla_{D_{e_i}e_i}\phi).
	\end{align*}
	For $\phi\in\Omega^1(\mathfrak{g}_E)$ and $\psi\in\Omega^2(\mathfrak{g}_E)$, define
	\begin{align*}
		\mathfrak{R}_g^\nabla(\phi)(X)&=\sum_{i=1}^n[R^\nabla(e_i,X),\phi(e_i)],\\
		\mathfrak{R}_g^\nabla(\psi)(X,Y)&=\sum_{i=1}^n[R^\nabla(e_i,X),\psi(e_i,Y)]-[R^\nabla(e_i,Y),\psi(e_i,X)].
	\end{align*}
	Then we have the following Bochner–Weizenböck formula first introduced by Bourguignon-Lawson.
	\begin{theorem}\cite{BL}
		For any $\phi\in\Omega^1(\mathfrak{g}_E)$ and $\psi\in\Omega^2(\mathfrak{g}_E)$, we have
		\begin{align*}
			&\Delta^\nabla\phi=\nabla^\ast\nabla+\phi\circ\operatorname{Ric}+\mathfrak{R}_g^\nabla(\phi),\\
			&\Delta^\nabla\psi=\nabla^\ast\nabla+\psi\circ(\operatorname{Ric}\wedge Id+2R_M)+\mathfrak{R}_g^\nabla(\psi),
		\end{align*}
		where $R_M$ is the curvature tensor of $M$ and
		\begin{itemize}
			\item $\operatorname{Ric} :TM\to TM$ is the Ricci transformation defined by
			\[\operatorname{Ric}\left( X \right) =\sum_jR_M(X,e_j)e_j ,\]
			\item $ \operatorname{Ric}\wedge \operatorname{Id} $ is the extension of the Ricci transformation $\operatorname{Ric}$ to $\wedge^2 TM$ given by
			\[ \operatorname{Ric}\wedge \operatorname{Id}(X,Y)=\operatorname{Ric}\wedge \operatorname{Id}(X\wedge Y)=\operatorname{Ric}(X)\wedge Y+X\wedge \operatorname{Ric}(Y),\]
			\item The composite map $\psi\circ R_M :\wedge^2 TM \to \Omega^0 \left( \mathfrak{g}_E \right)  $ is defined by
			\[\psi\circ R_M(X,Y)=\psi\circ R_M(X\wedge Y )  =\frac{1}{2}\sum_{j=1}^{n} \psi( e_j, R_M(X,Y)e_j) .\]
		\end{itemize}
	\end{theorem}
	\begin{remark}
		In particular, on the standard sphere $(S^n,g)$, we have
		\begin{align*}
			&\phi\circ\operatorname{Ric}=(n-1)\phi,\\
			&\psi\circ(\operatorname{Ric}\wedge Id+2R_{S^n})=2(n-2)\psi.
		\end{align*}
	\end{remark}
	\subsection{The conformal metric and the second variation}\ 
	
	Let $g$ be the standard metric on $S^n$, and assume $\tilde g=e^{2\varphi}g$ be the Riemannian metric on $S^n$ conformal to $g$ for some smooth map $\varphi\in C^\infty(S^n)$. Assume $D$ and $\tilde D$ be the Levi-Civita connections corresponding to $g$ and $\tilde g$ respectively. Then we have
	\begin{align*}
		\tilde D_XY=D_XY+X(\varphi)Y+Y(\varphi)X-g(X,Y)\operatorname{grad}(\varphi)
	\end{align*}
	for any $X,Y\in\mathscr{X}(S^n)$. For any connection $\nabla$ on $E$, assume $\nabla$ and $\widetilde\nabla$ are the induced connections on $\Omega^p(\mathfrak{g}_E)$ respect to $g$ and $\tilde g$ respectively, we have $R^{\tilde\nabla}=R^\nabla$ and
	\begin{align*}
		&d^{\tilde\nabla}\psi=d^\nabla\psi,\\
		&\delta^{\tilde\nabla}\psi=e^{-2\varphi}(\delta^\nabla\psi+(2p-n)i_{\operatorname{grad}(\varphi)}\psi)
	\end{align*}
	for any $\psi\in\Omega^p(\mathfrak{g}_E)$.
	
	The Yang-Mills functional on $(S^n,\tilde g)$ is given by
	\begin{align}\label{YM functional}
		\operatorname{YM}(\nabla)=\frac12\int_{S^n}|R^\nabla|^2_{\tilde g}dV_{\tilde g}=\frac12\int_{S^n}e^{(n-4)\varphi}|R^\nabla|^2_gdV_g.
	\end{align}
	The Euler-Lagrange equation is
	\begin{align}\label{YM in tilde g}
		\delta^{\tilde\nabla}R^\nabla=0,
	\end{align}
	or equivalently,
	\begin{align}\label{YM in g}
		\delta^\nabla R^\nabla=(n-4)i_{\operatorname{grad}(\varphi)}R^\nabla.
	\end{align}
	Assume $\nabla$ is a Yang-Mills connection satisfying (\ref{YM in g}). For any $B\in\Omega^1(\mathfrak{g}_E)$, let $\nabla^t=\nabla+tB$. Then we have
	\begin{align*}
		R^{\nabla^t}=R^\nabla+td^\nabla B+\frac12t^2[B\wedge B].
	\end{align*}
	The second variation of $\operatorname{YM}$ is given by
	\begin{align}\label{2nd var}
		\mathscr{L}(B):=\frac{d^2}{dt^2}\operatorname{YM}(\nabla^t)\mid_{t=0}=\int_{S^n}e^{(n-4)\varphi}\langle\mathscr{S}(B),B\rangle_gdV_g,
	\end{align}
	where
	\begin{align*}
		\mathscr{S}(B)=\delta^\nabla d^\nabla B-(n-4)i_{\operatorname{grad}(\varphi)}d^\nabla B+\mathfrak{R}_g^\nabla(B).
	\end{align*}
	For any gauge transformation $g\in\mathcal{G}$, $g$ acts on $\nabla$ such that $\nabla^g=g^{-1}\circ\nabla\circ g$. Then for any $\sigma\in\Omega^0(\mathfrak{g}_E)$, let $g_t=exp(t\sigma)$ be a family of gauge transformations. The variation along $\sigma$ is
	\begin{align*}
		\frac{d}{dt}\nabla^{g_t}\mid_{t=0}=d^\nabla\sigma.
	\end{align*}
	The Yang-Mills functional is invariant under gauge transformation. Therefore, we are interested in the variations orthogonal to gauge transformations respect to $\tilde g$, that is,
	\begin{align*}
		B\in\operatorname{Im}(d^\nabla)^\perp=\operatorname{ker}(\delta^{\tilde\nabla}).
	\end{align*}
	For any $B\in\operatorname{ker}(\delta^{\tilde\nabla})$, the second variation is given by
	\begin{align*}
		\mathscr{L}(B)=\int_{S^n}\langle\tilde{\mathscr{S}}(B),B\rangle_{\tilde g}dV_{\tilde g},
	\end{align*}
	where
	\begin{align*}
		\tilde{\mathscr{S}}(B)=\Delta^{\tilde\nabla}B+\mathfrak{R}^{\tilde\nabla}_{\tilde g}(B)
	\end{align*}
	is a elliptic, self-adjoint operator mapping $\operatorname{ker}(\delta^{\tilde\nabla})$ to itself. The eigenvalues of $\tilde{\mathscr{S}}$ on $\operatorname{ker}(\delta^{\tilde\nabla})$ are
	\begin{align*}
		\lambda_1<\lambda_2<...\to+\infty.
	\end{align*}
	Following (6.14) of \cite{BL}, we define a stable Yang-Mills connection.
	\begin{definition}\label{def:stable}
		We say a Yang-Mills connection is stable, if $\lambda_1>0$ and weakly stable if $\lambda_1\ge0$.
	\end{definition}
	It is readily verified that such defined weakly stable connections are equivalent to those with non-negative second variation.
	\section{The stability of Yang-Mills connections on conformal spheres}\ 
	
	In this section, we will prove Theorem \ref{thm 1} and Theorem \ref{thm 2}. As done in \cite{BL}, we construct the variational using gradient conformal vector fields of $(S^n,g)$ to prove the nonexistence of weakly stable Yang-Mills fields. These vector fields are precisely the gradients of the eigenfunctions corresponding to the first eigenvalue of the Laplace-Beltrami operator on $(S^n,g)$.
	\begin{proposition}\label{conformal vector field}
		For any vector $v\in\mathbb{R}^{n+1}$, let $F_v(x)=v\cdot x:\mathbb{R}^{n+1}\to\mathbb{R}$ be the linear map and $f_v=F_v\mid_{S^n}$ be the restriction on $S^n$. Then $V=\operatorname{grad}(f_v)$ is a conformal vector field of $(S^n,g)$ and satisfies
		\begin{align*}
			&D_XV=-f_vX,\\
			&D^\ast DV=V
		\end{align*}
		for any $X\in\mathscr{X}(S^n)$, where
		\begin{align*}
			D^\ast DV=-\sum_{i=1}^n(D_{e_i}D_{e_i}V-D_{D_{e_i}e_i}V).
		\end{align*}
	\end{proposition}
	
	To avoid any ambiguity, we designate the vector field corresponding to the vector $v$ as $V$. For some constant $\lambda\in\mathbb{R}$, let $\tilde V=e^{\lambda\varphi}V$, then we have
	\begin{align*}
		D_X\tilde V=-e^{\lambda\varphi}f_vX+\lambda X(\varphi)\tilde V.
	\end{align*}
	Inspired by \cite{BL}, for any $v\in\mathbb{R}^{n+1}$, we choose the variation of connection to be $B_v=i_{\tilde V}R^\nabla$. We need to compute the second variation that contains no derivatives of $\tilde V$.
	\begin{lemma}\label{S(iV_R)}
		For any $v\in\mathbb{R}^{n+1}$ and $X\in\mathscr{X}(S^n)$, we have
		\begin{align*}
			\mathscr{S}(B_v)(X)&=(4-n+\lambda\Delta\varphi-(\lambda^2+(n-4)\lambda)|\operatorname{grad}(\varphi)|_g^2)R^\nabla(\tilde V,X)+(4-n+3\lambda)e^{\lambda\varphi}f_vR^\nabla(\operatorname{grad}(\varphi),X)\\
			&-\lambda^2X(\varphi)R^\nabla(\operatorname{grad}(\varphi),\tilde V)-\lambda\nabla_{\operatorname{grad}(\varphi)}R^\nabla(\tilde V,X)-\lambda\nabla_{\tilde V}R^\nabla(\operatorname{grad}(\varphi),X)\\
			&+(n-4)R^\nabla(D_{\tilde V}\operatorname{grad}(\varphi),X)+\lambda R^\nabla(\tilde V,D_X\operatorname{grad}(\varphi)).
		\end{align*}
	\end{lemma}
	\begin{proof}
		For any $x\in S^n$, let $\{e_i\mid1\le i\le n\}$ be the orthogonal basis of $TS^n$ respect to $g$ such that $De_i(x)=0$. For any $e_j$, we have
		\begin{align*}
			&\delta^\nabla d^\nabla i_{\tilde V}R^\nabla(e_j)\\
			=&-\sum_{i=1}^n\nabla_{e_i}d^\nabla i_{\tilde V}R^\nabla(e_i,e_j)\\
			=&-\sum_{i=1}^n\nabla_{e_i}(\nabla_{e_i}i_{\tilde V}R^\nabla(e_j)-\nabla_{e_j}i_{\tilde V}R^\nabla(e_i))\\
			=&-\sum_{i=1}^n\nabla_{e_i}(\nabla_{e_i}(R^\nabla({\tilde V},e_j))-R^\nabla({\tilde V},D_{e_i}e_j)-\nabla_{e_j}(R^\nabla({\tilde V},e_i))+R^\nabla({\tilde V},D_{e_j}e_i))\\
			=&-\sum_{i=1}^n\nabla_{e_i}(\nabla_{e_i}R^\nabla({\tilde V},e_j)+R^\nabla(D_{e_i}{\tilde V},e_j)-\nabla_{e_j}R^\nabla({\tilde V},e_i)-R^\nabla(D_{e_j}{\tilde V},e_i))\\
			=&-\sum_{i=1}^n\nabla_{e_i}(\nabla_{\tilde V}R^\nabla(e_i,e_j)-2e^{\lambda\varphi}f_vR^\nabla(e_i,e_j)+\lambda e_i(\varphi)R^\nabla(\tilde V,e_j)-\lambda e_j(\varphi)R^\nabla(\tilde V,e_i))\\
			=&-\sum_{i=1}^n\nabla_{e_i}\nabla_{\tilde V}R^\nabla(e_i,e_j)+(2+\lambda\Delta\varphi-\lambda^2|\operatorname{grad}(\varphi)|_g^2)R^\nabla({\tilde V},e_j)+(8-2n+3\lambda)f_ve^{\lambda\varphi}R^\nabla(\operatorname{grad}(\varphi),e_j)\\
			&+\lambda(n-4-\lambda)e_j(\varphi)R^\nabla(\operatorname{grad}(\varphi),{\tilde V})-\lambda\nabla_{\operatorname{grad}(\varphi)}R^\nabla({\tilde V},e_j)+\lambda R^\nabla({\tilde V},D_{e_j}\operatorname{grad}(\varphi))
		\end{align*}
		at $x$, where we use $\sum_ie_i(f_v)e_i=\operatorname{grad}(f_v)=V$ and $-\sum_i\nabla_{e_i}R^\nabla(e_i,e_j)=\delta^\nabla R^\nabla(e_j)=(n-4)R^\nabla(\operatorname{grad}(\varphi),e_j)$. Employing the commutation formula, we have
		\begin{align*}
			&-\sum_{i=1}^n\nabla_{e_i}\nabla_{\tilde V}R^\nabla(e_i,e_j)\\
			=&-\sum_{i=1}^n([R^\nabla(e_i,{\tilde V}),R^\nabla(e_i,e_j)]-R^\nabla(R_{S^n}(e_i,{\tilde V})e_i,e_j)-R^\nabla(e_i,R_{S^n}(e_i,{\tilde V})e_j)+\nabla_{\tilde V}\nabla_{e_i}R^\nabla(e_i,e_j)+\nabla_{[e_i,{\tilde V}]}R^\nabla(e_i,e_j))\\
			=&-\mathfrak{R}_g^\nabla(i_{\tilde V}R^\nabla)(e_j)+(2-n)R^\nabla({\tilde V},e_j)-\sum_{i=1}^n(\nabla_{\tilde V}\nabla_{e_i}R^\nabla(e_i,e_j)+\nabla_{[e_i,{\tilde V}]}R^\nabla(e_i,e_j))
		\end{align*}
		at $x$, where
		\begin{align*}
			-\sum_{i=1}^n\nabla_{\tilde V}\nabla_{e_i}R^\nabla(e_i,e_j)=\nabla_{\tilde V}(\delta^\nabla R^\nabla(e_j))=(n-4)(\nabla_{\tilde V}R^\nabla(\operatorname{grad}(\varphi),e_j)+R^\nabla(D_{\tilde V}\operatorname{grad}(\varphi),e_j))
		\end{align*}
		and
		\begin{align*}
			-\sum_{i=1}^n\nabla_{[e_i,{\tilde V}]}R^\nabla(e_i,e_j)&=-\sum_{i=1}^n\nabla_{D_{e_i}{\tilde V}}R^\nabla(e_i,e_j)=(4-n)e^{\lambda\varphi}f_vR^\nabla(\operatorname{grad}(\varphi),e_j)-\lambda\nabla_{\tilde V}R^\nabla(\operatorname{grad}(\varphi),e_j).
		\end{align*}
		By analogous derivation, we obtain
		\begin{align*}
			&(4-n)i_{\operatorname{grad}(\varphi)}d^\nabla i_{\tilde V}R^\nabla(e_j)\\
			=&(4-n)(\nabla_{\tilde V}R^\nabla(\operatorname{grad}(\varphi),e_j)-2e^{\lambda\varphi}f_vR^\nabla(\operatorname{grad}(\varphi),e_j)+\lambda|\operatorname{grad}(\varphi)|_g^2R^\nabla({\tilde V},e_j)+\lambda e_j(\varphi)R^\nabla(\operatorname{grad}(\varphi),{\tilde V})).
		\end{align*}
		Synthesizing the preceding calculations, we conclude the proof of the lemma.
	\end{proof}
	\begin{proposition}\label{sum L(B_V_k)}
		Assume $\{v_1,...,v_{n+1}\}$ be an orthogonal basis of $\mathbb{R}^{n+1}$, then we have
		\begin{align}
			\sum_{k=1}^{n+1}\mathscr{L}(B_{v_k})=\int_{S^n}&e^{(n-4+2\lambda)\varphi}((\frac{4-n}2\Delta\varphi+2(\lambda+\frac{n-4}2)^2|\operatorname{grad}(\varphi)|_g^2+8-2n)|R^\nabla|_g^2\\
			&+\lambda(8-2n-\lambda)|i_{\operatorname{grad}(\varphi)}R^\nabla|^2)dV_g.
		\end{align}
	\end{proposition}
	\begin{proof}
		According to Lemma \ref{S(iV_R)}, we have
		\begin{align*}
			\mathscr{L}(B_v)=(4-n)\int_{S^n}e^{(n-4)\varphi}q(v,v)dV_g,
		\end{align*}
		where at any $x\in S^n$, $q$ is a quadratic form on $\mathbb{R}^{n+1}$ defined by
		\begin{align*}
			q(v,w)=\sum_{j=1}^n&\langle (4-n+\lambda\Delta\varphi-(\lambda^2+(n-4)\lambda)|\operatorname{grad}(\varphi)|_g^2)R^\nabla(\tilde V,e_j)+(4-n+3\lambda)e^{\lambda\varphi}f_vR^\nabla(\operatorname{grad}(\varphi),e_j)\\
			&-\lambda^2e_j(\varphi)R^\nabla(\operatorname{grad}(\varphi),\tilde V)-\lambda\nabla_{\operatorname{grad}(\varphi)}R^\nabla(\tilde V,e_j)-\lambda\nabla_{\tilde V}R^\nabla(\operatorname{grad}(\varphi),e_j)\\
			&+(n-4)R^\nabla(D_{\tilde V}\operatorname{grad}(\varphi),e_j)+\lambda R^\nabla(\tilde V,D_{e_j}\operatorname{grad}(\varphi)),R^\nabla(\tilde W,e_j)\rangle
		\end{align*}
		For any $x\in S^n$, note that $\sum_{k=1}^{n+1}q(v_k,v_k)(x)$ is independent to the choice of $\{v_k\}$, we may assume $v_{n+1}=x$ and $\{v_k\mid1\le k\le n\}$ is an orthogonal basis of $T_xS^n\subset T_x\mathbb{R}^{n+1}$. From proposition \ref{conformal vector field}, we have $V_k=e_k$ for $1\le k\le n$ and $V_{n+1}=0$ at $x$. Thus we have
		\begin{align*}
			\sum_{k=1}^{n+1}\sum_{j=1}^n\langle R^\nabla(\tilde V_k,e_j),R^\nabla(\tilde V_k,e_j)\rangle=e^{2\lambda\varphi}\sum_{i,j=1}^n\langle R^\nabla(e_i,e_j),R^\nabla(e_i,e_j)\rangle=2e^{2\lambda\varphi}|R^\nabla|_g^2.
		\end{align*}
		Since $\sum_kf_{v_k}^2\equiv1$ is a constant map, we have
		\begin{align*}
			\sum_{k=1}^{n+1}f_{v_k}\tilde V_k=\frac12e^{\lambda\varphi}\operatorname{grad}(\sum_{k=1}^{n+1}f_v^2)=0,
		\end{align*}
		and thus
		\begin{align*}
			\sum_{k=1}^{n+1}\sum_{j=1}^n\langle e^{\lambda\varphi}f_vR^\nabla(\operatorname{grad}(\varphi),e_j),R^\nabla(\tilde V_k,e_j)\rangle=0.
		\end{align*}
		We also have
		\begin{align*}
			\sum_{k=1}^{n+1}\sum_{j=1}^n\langle e_j(\varphi)R^\nabla(\operatorname{grad}(\varphi),\tilde V_k),R^\nabla(\tilde V_k,e_j)\rangle=-e^{2\lambda\varphi}|i_{\operatorname{grad}(\varphi)}R^\nabla|_g^2.
		\end{align*}
		We assume $De_i=0$ at $x$. Note that
		\begin{align*}
			&\sum_{i,j=1}^n\langle\nabla_{e_i}R^\nabla(\operatorname{grad}(\varphi),e_j),R^\nabla(e_i,e_j)\rangle\\
			=&-\sum_{i,j=1}(\langle\nabla_{\operatorname{grad}(\varphi)}R^\nabla(e_j,e_i),R^\nabla(e_i,e_j)\rangle+\langle\nabla_{e_j}R^\nabla(e_i,\operatorname{grad}(\varphi)),R^\nabla(e_i,e_j)\rangle\\
			=&\operatorname{grad}(\varphi)(|R^\nabla|_g^2)-\sum_{i,j=1}^n\langle\nabla_{e_i}R^\nabla(\operatorname{grad}(\varphi),e_j),R^\nabla(e_i,e_j)\rangle.
		\end{align*}
		Thus we have
		\begin{align*}
			\sum_{k=1}^{n+1}\sum_{j=1}^n\langle \nabla_{\tilde V_k}R^\nabla(\operatorname{grad}(\varphi),e_j),R^\nabla(\tilde V_k,e_j)\rangle=e^{2\lambda\varphi}\sum_{i,j=1}^n\langle\nabla_{e_i}R^\nabla(\operatorname{grad}(\varphi),e_j),R^\nabla(e_i,e_j)\rangle=\frac12e^{2\lambda\varphi}\operatorname{grad}(\varphi)(|R^\nabla|_g^2).
		\end{align*}
		Still assuming $De_i=0$ at $x$, we have
		\begin{align*}
			&\sum_{k=1}^{n+1}\sum_{j=1}^n\langle (n-4)R^\nabla(D_{\tilde V_k}\operatorname{grad}(\varphi),e_j)+\lambda R^\nabla(\tilde V_k,D_{e_j}\operatorname{grad}(\varphi)),R^\nabla(\tilde V_k,e_j)\rangle\\
			=&(n-4+\lambda)e^{2\lambda\varphi}\sum_{i,j=1}^n\langle R^\nabla(D_{e_i}\operatorname{grad}(\varphi),e_j),R^\nabla(e_i,e_j)\rangle
		\end{align*}
		and
		\begin{align*}
			&\sum_{i,j=1}^n\langle R^\nabla(D_{e_i}\operatorname{grad}(\varphi),e_j),R^\nabla(e_i,e_j)\rangle\\
			=&\sum_{i,j=1}^ne_i(\langle R^\nabla(\operatorname{grad}(\varphi),e_j),R^\nabla(e_i,e_j)\rangle)-\langle\nabla_{e_i}R^\nabla(\operatorname{grad}(\varphi),e_j),R^\nabla(e_i,e_j)\rangle-\langle R^\nabla(\operatorname{grad}(\varphi),e_j),\nabla_{e_i}R^\nabla(e_i,e_j)\rangle\\
			=&\operatorname{div}(\sum_{j=1}^n\langle R^\nabla(\operatorname{grad}(\varphi),e_j),R^\nabla(\cdot,e_j)\rangle)-\frac12\operatorname{grad}(\varphi)(|R^\nabla|_g^2)+(n-4)|i_{\operatorname{grad}(\varphi)}R^\nabla|_g^2.
		\end{align*}
		By direct computation, we obtain
		\begin{align*}
			-\int_{S^n}e^{(n-4+2\lambda)\varphi}\operatorname{div}(\sum_{j=1}^n\langle R^\nabla(\operatorname{grad}(\varphi),e_j),R^\nabla(\cdot,e_j)\rangle)dV_g=(n-4+2\lambda)\int_{S^n}e^{(n-4+2\lambda)\varphi}|i_{\operatorname{grad}(\varphi)}R^\nabla|_g^2dV_g
		\end{align*}
		and
		\begin{align*}
			\frac12\int_{S^n}e^{(n-4+2\lambda)\varphi}\operatorname{grad}(\varphi)(|R^\nabla|_g^2)dV_g=\frac12\int_{S^n}e^{(n-4+2\lambda)\varphi}(\Delta\varphi-(n-4+2\lambda)|\operatorname{grad}(\varphi)|_g^2)|R^\nabla|_g^2dV_g.
		\end{align*}
		Thus, we finish the proof.
	\end{proof}
	From the lemma above and letting $\lambda=0$, we can immediately prove the Theorem \ref{thm 1} on $S^n$
	\begin{theorem}
		Assume $(S^n,g)$ is the standard sphere and $\tilde g=e^{2\varphi}g$ is a conformal metric. If $n\ge5$ and
		\begin{align*}
			\frac12\Delta\varphi+\frac{4-n}2|\operatorname{grad}(\varphi)|_g^2+2>0,
		\end{align*}
		then there is no weakly stable Yang-Mills connections on $(S^n,\tilde g)$.
	\end{theorem}
	Note that
	\begin{align*}
		|i_{\operatorname{grad}(\varphi)}R^\nabla(e_j)|\le\sum_{i=1}^n|e_i(\varphi)||R^\nabla(e_i,e_j)|\le|\operatorname{grad}(\varphi)|_g(\sum_{i=1}^n|R^\nabla(e_i,e_j)|^2)^{\frac12}
	\end{align*}
	for any $e_j$, thus
	\begin{align*}
		|i_{\operatorname{grad}(\varphi)}R^\nabla|_g^2=\sum_{j=1}^n|i_{\operatorname{grad}(\varphi)}R^\nabla(e_j)|^2\le|\operatorname{grad}(\varphi)|_g^2\sum_{i,j=1}^n|R^\nabla(e_i,e_j)|^2=2|\operatorname{grad}(\varphi)|_g^2|R^\nabla|_g^2.
	\end{align*}
	By direct calculation, we have
	\begin{align*}
		\tilde D_X\tilde V=-e^{\lambda\varphi}X+(\lambda+1)X(\varphi)\tilde V+\tilde V(\varphi)X-\tilde g(X,\tilde V)\widetilde{\operatorname{grad}}(\varphi).
	\end{align*}
	For any $x\in S^n$, assume $\{\tilde e_i\}$ be an orthogonal basis of $TS^n$ respect to $\tilde g$ with $\tilde D\tilde e_i(x)=0$. Then at $x$ we have
	\begin{align*}
		\delta^{\tilde\nabla}i_{\tilde V}R^\nabla&=-\sum_{i=1}^n\tilde\nabla_{\tilde e_i}i_{\tilde V}R^\nabla(\tilde e_i)=-\sum_{i=1}^n\nabla_{\tilde e_i}(R^\nabla(\tilde V,\tilde e_i))=-\delta^{\tilde\nabla}R^\nabla(\tilde V)-\sum_{i=1}^nR^\nabla(\tilde D_{\tilde e_i}\tilde V,\tilde e_i)\\
		&=-(\lambda+2)R^\nabla(\tilde V,\widetilde{\operatorname{grad}}(\varphi)).
	\end{align*}
	Thus $e^{-2\varphi}i_VR^\nabla\in\operatorname{ker}(\delta^{\tilde\nabla})$ for any $v\in\mathbb{R}^{n+1}$. Taking $\lambda=-2$, we obtain
	\begin{align*}
		\sum_{k=1}^n\mathscr{L}(B_{v_k})\le\int_{S^n}e^{(n-8)\varphi}(\frac{4-n}2\Delta\varphi+\frac{n^2-16}2|\operatorname{grad}(\varphi)|_g^2+8-2n)|R^\nabla|_g^2dV_g.
	\end{align*}
	Then we can prove Theorem \ref{thm 2}
	\begin{theorem}
		Assume $(S^n,g)$ is the standard sphere and $\tilde g=e^{2\varphi}g$ is a conformal metric. If
		\begin{align*}
			\frac{4-n}2\Delta\varphi+\frac{n^2-16}2|\operatorname{grad}(\varphi)|_g^2+8-2n\le0,
		\end{align*}
		then there is no stable Yang-Mills connections on $(S^n,\tilde g)$.	
	\end{theorem}
	\section{The stability of Yang-Mills connections on $S^{n_1}\times...\times S^{n_q}$}\ 
	
	In this section, we will prove Theorem \ref{thm:S^n times S^m}. Assume $M=S^{n_1}\times...\times S^{n_q}$ is a product space with the product metric $g=g_{S^{n_1}}\oplus...\oplus g_{S^{n_q}}$, where $g_{S^{n_k}}$ is the standard sphere metric in $S^{n_k}$. Let $D^k$ be the Levi-Civita connection on $S^{n_k}$, then for any $x=(x_1,...,x_q)\in M$ and $X=\sum_kX_k\in T_xM=\oplus_k T_{x_k}S^{n_k}$, we have
	\begin{align*}
		DX=\sum_{k=1}^qD^kX_k.
	\end{align*}
	We still construct the variation by using gradient conformal vector fields on each component $S^{n_k}$. Such vector fields have the following properties.
	\begin{proposition}
		For any vector $v^k\in\mathbb{R}^{n_k+1}$, let $F_{v^k}(x)=v^k\cdot x$ be the linear map and $f_{v^k}=F_{v^k}\mid_{S^{n_k}}$ be the restriction on $S^{n_k}$. Then $V^k:=\operatorname{grad}_{S^{n_k}}(f_{v^k})$ satisfies
		\begin{align*}
			(1)\ &D_XV^k=-f_{v^k}X\textrm{ for any }X\in\mathscr{X}(S^{n_k}),\\
			(2)\ &D_XV^k=0\textrm{ for any }X\in\mathscr{X}(S^{n_p})\ (p\ne k),\\
			(3)\ &D^\ast DV^k=V^k.
		\end{align*}
	\end{proposition}
	To avoid any ambiguity, we designate the vector field corresponding to the vector $v^k$ as $V^k$. For any $x=(x_1,...,x_q)\in M$, assume $\{e_i^p\}$ be an orthogonal basis of $T_{x_p}S^{n_p}\subset T_xM$. If $\nabla$ is a Yang-Mills connection satisfying (\ref{YM equation}), by direct calculation we have
	\begin{align*}
		\delta^\nabla i_{V^k}R^\nabla=-\sum_{p=1}^q\sum_{i=1}^{n_p}\nabla_{e_i^p}i_{V^k}R^\nabla=-\delta^\nabla R^\nabla(V^k)-\sum_{p=1}^q\sum_{i=1}^{n_p}R^\nabla(D_{e_i^p}V^k,e_i^p)=0.
	\end{align*}
	Thus the second variation along $i_{V^k}R^\nabla$ is
	\begin{align*}
		\mathscr{L}(i_{V^k}R^\nabla)=\int_M\langle\mathscr{S}(i_{V^k}R^\nabla),i_{V^k}R^\nabla\rangle dV,
	\end{align*}
	where
	\begin{align*}
		\mathscr{S}(i_{V^k}R^\nabla)=\Delta^\nabla i_{V^k}R^\nabla+\mathfrak{R}^\nabla_g(i_{V^k}R^\nabla).
	\end{align*}
	\begin{lemma}\label{L(i_VR) in S^m times S^n}
		For any vector $v^k\in\mathbb{R}^{n_k+1}$, assume $V^k$ be the vector field defined above, then we have
		\begin{align*}
			\mathscr{L}(i_{V^k}R^\nabla)=\int_M\sum_{p=1}^q\sum_{i=1}^{n_p}(2+2\delta_{pk}-n_p)|R^\nabla(V^k,e_i^p)|^2+2f_{v^k}\sum_{p=1}^q\sum_{j=1}^{n_p}\sum_{i=1}^{n_k}\langle\nabla_{e_i^k}R^\nabla(e_i^k,e_j^p),R^\nabla(V^k,e_j^p)\rangle dV_g.
		\end{align*}
	\end{lemma}
	\begin{proof}
		Using Bochner-Weizenb\"{o}ck formula, we have
		\begin{align*}
			\Delta^\nabla i_{V^k}R^\nabla+\mathfrak{R}^\nabla_g(i_{V^k}R^\nabla)=\nabla^\ast\nabla i_{V^k}R^\nabla+2\mathfrak{R}^\nabla(i_{V^k}R^\nabla)+i_{V^k}R^\nabla\circ\operatorname{Ric}_M.
		\end{align*}
		Assume $\{e_i^p\mid 1\le i\le n_p\}$ be an orthogonal basis of $TS^{n_p}$. Then we have $\operatorname{Ric}_M(e_i^p)=(n_p-1)e_i^p$, and thus
		\begin{align*}
			\langle i_{V^k}R^\nabla\circ \operatorname{Ric}_M,i_{V^k}R^\nabla\rangle=\sum_{p=1}^q\sum_{i=1}^{n_p}(n_p-1)|R^\nabla(V^k,e_i^p)|^2.
		\end{align*}
		For any $x=(x_1,...,x_q)\in M$, we additionally assume that $D^pe_i^p=0$ at $x_p$. Then at $x$ we have
		\begin{align*}
			\nabla^\ast\nabla i_{V^k}R^\nabla(e_j^p)&=-\sum_{s=1}^q\sum_{i=1}^{n_s}\nabla_{e_i^s}\nabla_{e_i^s}i_{V^k}R^\nabla(e_j^p)=-\sum_{s=1}^q\sum_{i=1}^{n_s}\nabla_{e_i^s}(\nabla_{e_i^s}i_{V^k}R^\nabla(e_j^p))\\
			&=-\sum_{s=1}^q\sum_{i=1}^{n_s}\nabla_{e_i^s}(\nabla_{e_i^s}R^\nabla(V^k,e_j^p)+R^\nabla(D_{e_i^s}V^k,e_j^p))\\
			&=\nabla^\ast\nabla R^\nabla(V^k,e_j^p)+R^\nabla(D^\ast DV^k,e_j^p)-2\sum_{s=1}^q\sum_{i=1}^{n_s}\nabla_{e_i^s}R^\nabla(D_{e_i^s}V^k,e_j^p)\\
			&=\nabla^\ast\nabla R^\nabla(V^k,e_j^p)+R^\nabla(V^k,e_j^p)+2f_{v^k}\sum_{i=1}^{n_k}\nabla_{e_i^k}R^\nabla(e_i^k,e_j^p).
		\end{align*}
		Yang-Mills equation (\ref{YM equation}) and Bianchi identity $d^\nabla R^\nabla=0$ shows that $\Delta^\nabla R^\nabla
		=0$. Using Bochner-Weizenb\"{o}ck formula again, we have
		\begin{align*}
			\nabla^\ast\nabla R^\nabla(V^k,e_j^p)=-\mathfrak{R}^\nabla(R^\nabla)(V^k,e_j^p)-R^\nabla\circ(\operatorname{Ric}_M\wedge I+2R_M)(V^k,e_j^p).
		\end{align*}
		It is easily to check that
		\begin{align*}
			-\mathfrak{R}^\nabla(R^\nabla)(V^k,e_j^p)=-2\mathfrak{R}^\nabla(i_{V^k}R^\nabla)(e_j^p).
		\end{align*}
		Using $\operatorname{Ric}_M(V^k)=(n_k-1)V^k$, we have
		\begin{align*}
			-\langle i_{V^k}R^\nabla\circ(\operatorname{Ric}_M\wedge I),i_{V^k}R^\nabla\rangle=\sum_{p=1}^q\sum_{i=1}^{n_p}(2-n_k-n_p)|R^\nabla(V^k,e_i^p)|^2.
		\end{align*}
		Since $R_M(V^k,e_i^p)e_j^{p'}=0$ for $p\ne k$ or $p'\ne k$ and $R_M(V^k,e_i^k)e_j^k=R_{S^{n_k}}(V^k,e_i^k)e_j^k$, we have
		\begin{align*}
			-\langle i_{V^k}R^\nabla\circ(2R_M),i_{V^k}R^\nabla\rangle=-\sum_{i,j=1}^{n_k}\langle R^\nabla(e_j^k,R_{S^{n_k}}(V^k,e_i^k)e_j^k),R^\nabla(V^k,e_j^k)\rangle=2\sum_{i=1}^{n_k}|R^\nabla(V^k,e_i^k)|^2.
		\end{align*}
		Then we finish the proof.
	\end{proof}
	\begin{proposition}
		Assume $\{v_1^k,...,v_{n_k}^k\}$ be an orthogonal basis of $\mathbb{R}^{n_k+1}$, then we have
		\begin{align*}
			\sum_{k=1}^q\sum_{l=1}^{n_k+1}\mathscr{L}(i_{V_l^k}R^\nabla)=\int_M\sum_{p,k=1}^q\sum_{i=1}^{n_p}\sum_{j=1}^{n_k}(2+2\delta_{pk}-n_p)|R^\nabla(e_j^k,e_i^p)|^2dV_g.
		\end{align*}
	\end{proposition}
	\begin{proof}
		Since $\sum_{l=1}^{n_k+1}f_{v^k_l}^2\equiv1$ on $S^{n_k}$ for any $1\le k\le q$, we have
		\begin{align*}             \sum_{l=1}^{n_k+1}f_{v^k_l}V^k_l=\frac12\operatorname{grad}_{S^{n_k}}(\sum_{l=1}^{n_k+1}f_{v^k_l}^2)=0,
		\end{align*}
		and thus the sum of the second term in lemma \ref{L(i_VR) in S^m times S^n} vanishes. To calculate the first term, note that the the trace of $q_{k,i,p}$ is independent to the chioce of $\{v^k_l\}$ for fixed $(k,i,p)$ and $x\in M$, where
		\begin{align*}
			q_{k,i,p}(v^k_l,w^k_l):=\langle R^\nabla(V^k_l,e_i^p),R^\nabla(W^k_l,e^i_p)\rangle(x).
		\end{align*}
		For any $x=(x_1,...,x_q)\in M$, we may assume $v^k_{n_k+1}=x_k$ and $\{v^k_1,...,v^k_{n_k}\}$ is an orthogonal basis of tangent vectors of $S^{n_k}\subset\mathbb{R}^{n_k+1}$ at $x_k$. Then we have $V^k_{n_k+1}=0$ at $x$ and $\{V_1^k,...,V^k_{n_k}\}$ is an orthogonal basis of $T_xS^{n_k}$. Hence we have
		\begin{align*}
			\sum_{l=1}^{n_l+1}|R^\nabla(V^k_l,e_i^p)|^2=\sum_{l=1}^{n_l}|R^\nabla(e^k_l,e_i^p)|^2
		\end{align*}
		for any fixed $(k,i,p)$. Then we finish the proof.
	\end{proof}
	The above proposition directly implies Theorem \ref{thm:S^n times S^m}
	\begin{theorem}
		(1)\ If $n_k\ge5$ for any $1\le k\le q$, then there is no nontrivial weakly stable Yang-Mills on $M$.\\
		(2)\ If $n_k\ge4$ for any $1\le k\le q$, then there is no nontrivial stable Yang-Mills on $M$.
	\end{theorem}
	\section{The stability of Yang-Mills connections on non-compact manifolds with warped product metric}\ 
	
	In this section, we will prove Theorem \ref{thm 3}. Assume $M^n=I\times N^{n-1}$ is a Riemannian manifold with warped product metric $g_M=drdr+f^2(r)g_N$, where $I\subset\mathbb{R}$ is an open interval and $f\in C^2(I,(0,+\infty))$ satisties conditions (a)-(d) in (\ref{condition for f}). Let $D$ and $D'$ are the Levi-Civita connections of $M$ and $N$ respectively. For any $(r,x')\in M$ let $(y^1,...,y^{n-1})$ be the local normal coordinates of $x'\in N$. Define vector fields $T=\frac{\partial}{\partial r}$ and $Y_i=\frac{\partial}{\partial y^i}$ for $i\in[1,n-1]$, and it is easily to verify that
	\begin{equation}\label{L-C connections on w-p mfd}
		\begin{split}
			&D_TT=0,\\
			&D_TY_i=D_{Y_i}T=f^{-1}f'Y_i,\\
			&D_{Y_i}Y_j=\iota_{r\ast}(D'_{Y_i}Y_j)-f^{-1}f'g(Y_i,Y_j)T,
		\end{split}
	\end{equation}
	where $\iota_r:N\to M$ is the slice embedding such that $\iota_r(x')=(r,x')$. Then $V=f(r)\frac{\partial}{\partial r}$ is a conformal vector field on $M$ satisfying
	\begin{align*}
		D_XV=f'(r)X
	\end{align*}
	for any $X\in\mathscr{X}(M)$.
	
	Let $\nabla$ be a Yang-Mills connection and $\nabla^t=\nabla+tB$ be a family of connections for any $B\in\Omega^1(\mathfrak{g}_E)$. The second variation of Yang-Mills functional on $M$ is given by
	\begin{align*}
		\mathscr{L}(B):=\frac{d^2}{dt^2}YM(\nabla^t)\mid_{t=0}=\int_M|d^\nabla B|^2+\langle\mathfrak{R}^\nabla_g(B),B\rangle dV_g.
	\end{align*}
	Moreover, if we additionally assume that $B$ has compact support, then
	\begin{align*}
		\mathscr{L}(B)=\int_M\langle\delta^\nabla d^\nabla B+\mathfrak{R}^\nabla(B),B\rangle dV_g.
	\end{align*}
	\begin{lemma}\label{S(i_VR) in W-P}
		Let $\nabla$ be a Yang-Mills connection, then we have
		\begin{align*}
			\delta^\nabla d^\nabla i_VR^\nabla+\mathfrak{R}^\nabla_g(i_VR^\nabla)=(n-4)f^{-1}f''i_VR^\nabla.
		\end{align*}
	\end{lemma}
	\begin{proof}
		For any $x=(r,x')\in M$, we assume $\{e_1,...,e_n\}$ be an orthogonal basis of $TM$ such that $De_i=0$ at $x$. By direct calculation, we have
		\begin{align*}
			&\delta^\nabla d^\nabla i_VR^\nabla(e_j)\\
			=&-\sum_{i=1}^n\nabla_{e_i}(d^\nabla i_VR^\nabla(e_i,e_j))\\
			=&-\sum_{i=1}^n\nabla_{e_i}(\nabla_{e_i}i_VR^\nabla(e_j)-\nabla_{e_j}i_VR^\nabla(e_i))\\
			=&-\sum_{i=1}^n\nabla_{e_i}(\nabla_{e_i}R^\nabla(V,e_j)+R^\nabla(D_{e_i}V,e_j)-\nabla_{e_j}R^\nabla(V,e_i)-R^\nabla(D_{e_j}V,e_i))\\
			=&-\sum_{i=1}^n\nabla_{e_i}(\nabla_VR^\nabla(e_i,e_j)+2f'(r)R^\nabla(e_i,e_j))\\
			=&-\sum_{i=1}^n(\nabla_{e_i}\nabla_VR^\nabla(e_i,e_j)+2f'\nabla_{e_i}R^\nabla(e_i,e_j))-2f''R^\nabla(\operatorname{grad}(f'),e_j).
		\end{align*}
		Using $\operatorname{grad}(f')=f''T=f^{-1}f''V$ and $\sum_i\nabla_{e_i}R^\nabla(e_i,e_j)=-\delta^\nabla R(e_j)=0$, we have
		\begin{align*}
			\delta^\nabla d^\nabla i_VR^\nabla(e_j)=-\sum_{i=1}^n\nabla_{e_i}\nabla_VR^\nabla(e_i,e_j)-2f^{-1}f''R^\nabla(V,e_j).
		\end{align*}
		Applying the commutation formula and note that $[e_i,V]=D_{e_i}V=f'e_i$, we have
		\begin{align*}
			&-\sum_{i=1}^n\nabla_{e_i}\nabla_VR^\nabla(e_i,e_j)\\
			=&-\sum_{i=1}^n[R^\nabla(e_i,V),R^\nabla(e_i,e_j)]+\nabla_V\nabla_{e_i}R^\nabla(e_i,e_j)+\nabla_{[e_i,V]}R^\nabla(e_i,e_j)-R^\nabla(R_M(e_i,V)e_i,e_j)-R^\nabla(e_i,R_M(e_i,V)e_j)\\
			=&-\mathfrak{R}^\nabla(i_VR^\nabla)(e_j)+\nabla_V(\delta^\nabla R^\nabla(e_j))+f'\delta^\nabla R^\nabla(e_j)+\sum_{i=1}^nR^\nabla(R_M(e_i,V)e_i,e_j)+R^\nabla(e_i,R_M(e_i,V)e_j)\\
			=&-\mathfrak{R}^\nabla(i_VR^\nabla)(e_j)+\sum_{i=1}^nR^\nabla(R_M(e_i,V)e_i,e_j)+R^\nabla(e_i,R_M(e_i,V)e_j).
		\end{align*}
		Finally we need to calculate the curvature tensor $R_M$. Let $(y^1,...,y^{n-1})$ be the local normal coordinate of $x'\in N$ and $T,Y_i\in\mathscr{X}(M)$ be the vector fields defined in (\ref{L-C connections on w-p mfd}). Since the above equation is independent to the choice of $\{e_i\}$, we may assume $e_i=f^{-1}Y_i$ for $i\in[1,n-1]$ and $e_n=T$. Using (\ref{L-C connections on w-p mfd}), at $x$ we have
		\begin{align*}
			D_{Y_i}D_VY_j=D_{Y_i}(f'Y_j)=f'\iota_{r\ast}(D'_{Y_i}Y_j)-\delta_{ij}f(f')^2T=f'\iota_{r\ast}(D'_{Y_i}Y_j)-\delta_{ij}(f')^2V,
		\end{align*}
		and
		\begin{align*}
			D_VD_{Y_i}Y_j&=fD_T(\iota_{r\ast}(D'_{Y_i}Y_j)-f^{-1}f'g(Y_i,Y_j)T)\\
			&=-\delta_{ij}f^2(f^{-1}f')'V-f^{-1}f'(g(D_TY_i,Y_j)+g(Y_i,D_TY_j))V\\
			&=-\delta_{ij}(ff''+(f')^2)V,
		\end{align*}
		where we use $D_T(\iota_{r\ast}(D'_{Y_i}Y_j))=0$, since by letting $\Gamma'$ be the Christoffel symbols of $N$ respect to $(y^1,...,y^{n-1})$, we have $\Gamma_{ij}'^k(x')=0$ and thus
		\begin{align*}
			D_T(\iota_{r\ast}(D'_{Y_i}Y_j))(x)=\sum_{k=1}^{n-1}\Gamma_{ij}'^k(x')D_TY_k(x)=0.
		\end{align*}
		Finally, we have $[Y_i,V]=D_{Y_i}V-D_VY_i=0$, and thus
		\begin{align*}
			D_{[Y_i,V]}Y_j=0.
		\end{align*}
		Hence we obtain
		\begin{align*}
			R_M(e_i,V)e_j=f^{-2}R_M(Y_i,V)Y_j=\delta_{ij}f^{-1}f''V
		\end{align*}
		for any $i,j\in[1,n-1]$. Then we have
		\begin{align*}
			\sum_{i=1}^nR^\nabla(R_M(e_i,V)e_i,e_j)+R^\nabla(e_i,R_M(e_i,V)e_j)=(n-2)f^{-1}f''R^\nabla(V,e_j).
		\end{align*}
		Using similar method, we can proof that
		\begin{align*}
			R_M(e_i,V)T=-f''e_i
		\end{align*}
		for any $i\in[1,n-1]$ and hence
		\begin{align*}
			\sum_{i=1}^nR^\nabla(R_M(e_i,V)e_i,T)+R^\nabla(e_i,R_M(e_i,V)T)=0.
		\end{align*}
		Thus, we finish the proof.
	\end{proof}
	Since $i_VR^\nabla$ is not a 1-form with compact support, we also need the following lemma.
	\begin{lemma}\label{energy limit}
		Assume $\nabla$ is a weakly stable Yang-Mills connection on $M$ with $R^\nabla\in L^\infty(M)\cap L^2(M)$. Then we have
		\begin{align*}
			\liminf_{R\to+\infty}\int_{(\frac2R,R)\times N}\langle \delta^\nabla d^\nabla i_VR^\nabla+\mathfrak{R}_g^\nabla(i_VR^\nabla),i_VR^\nabla\rangle dV_g\ge0
		\end{align*}
	\end{lemma}
	\begin{proof}
		We only need to prove the case where $I=(0,+\infty)$, as other cases can be handled by analogous arguments. For any $R>0$, let $\eta_R\in C^\infty([0,+\infty))$ be a cut-off function such that $\eta_R(r)\in[0,1]$ and
		\begin{align*}
			&\eta_R=1\textrm{ on }(\frac2R,R),\\
			&\eta_R=0\textrm{ on }[0,\frac1R)\cup(2R,+\infty),\\
			&|d^k\eta_R|\le C_kR^{-k}\textrm{ on }(R,2R),\\
			&|d^k\eta_R|\le C_kR^k\textrm{ on }(\frac1R,\frac2R)
		\end{align*}
		for any $k\in\mathbb{N}$ and some constant $C_k>0$. For any $x\in M$, assume $D e_i=0$ at $x$, then we have
		\begin{align*}
			&\delta^\nabla d^\nabla(\eta_Ri_VR^\nabla)(e_j)\\
			=&-\sum_{i=1}^n\nabla_{e_i}(\nabla_{e_i}(\eta_Ri_VR^\nabla)(e_j)-\nabla_{e_j}(\eta_Ri_VR^\nabla)(e_i))\\
			=&-\sum_{i=1}^n\nabla_{e_i}(\eta_Rd^\nabla i_VR^\nabla(e_i,e_j)+e_i(\eta_R)i_VR^\nabla(e_j)-e_j(\eta_R)i_VR^\nabla(e_i))\\
			=&\eta_R\delta^\nabla d^\nabla i_VR^\nabla(e_j)-\eta_R''i_VR^\nabla(e_j)-\eta_R'\nabla_Ti_VR^\nabla(e_j)-\eta_R' d^\nabla i_VR^\nabla(T,e_j)
		\end{align*}
		at $x$. By direct calculation, we have
		\begin{align*}
			\nabla_Ti_VR^\nabla=d^\nabla i_VR^\nabla(T,e_j)=\nabla_TR^\nabla(V,e_j)+f^{-1}f'R^\nabla(V,e_j).
		\end{align*}
		Thus we have
		\begin{align*}
			\delta^\nabla d^\nabla(\eta_Ri_VR^\nabla)(e_j)=\eta_R\delta^\nabla d^\nabla i_VR^\nabla(e_j)-(\eta_R''+2\eta_R'f^{-1}f')i_VR^\nabla(e_j)-2\eta_R'\nabla_TR^\nabla(V,e_j).
		\end{align*}
		Since $\nabla$ is weakly stable, we have
		\begin{align*}
			&0\le\mathscr{L}(\eta_R(r)\cdot i_VR^\nabla)=\int_M\langle\delta^\nabla d^\nabla(\eta_R(r)\cdot i_VR^\nabla)+\mathfrak{R}^\nabla_g(\eta_R(r)\cdot i_VR^\nabla),\eta_R(r)\cdot i_VR^\nabla\rangle dV_g\\
			=&\int_M\eta_R^2(r)\langle\delta^\nabla d^\nabla i_VR^\nabla+\mathfrak{R}^\nabla_g(i_VR^\nabla),i_VR^\nabla\rangle-\eta_R(\eta_R''+2\eta_R'f^{-1}f')|i_VR^\nabla|^2-\eta_R\eta_R'T(|i_VR^\nabla|^2) dV_g\\
			=&\int_{(\frac2R,R)\times N}\langle\delta^\nabla d^\nabla i_VR^\nabla+\mathfrak{R}^\nabla_g(i_VR^\nabla),i_VR^\nabla\rangle dV_g\\&+\int_{((\frac1R,\frac2R)\cup(R,2R))\times N}((n-4)\eta_R^2f^{-1}f''+\eta_R'^2-2\eta_R\eta_R'f^{-1}f')|i_VR^\nabla|^2dV_g.
		\end{align*}
		Since $|i_VR^\nabla|^2\le 2f^2|R^\nabla|^2$ and $R^\nabla\in L^2(M)$, using $ff''\in L^\infty((0,+\infty))$ and $f(|f'|+1)=O(r)$ as $r\to+\infty$, we have
		\begin{align*}
			\int_{(R,2R)\times N}((n-4)\eta_R^2f^{-1}f''+\eta_R'^2-2\eta_R\eta_R'f^{-1}f')|i_VR^\nabla|^2dV_g\le C\int_{(R,2R)\times N}|R^\nabla|^2dV_g\to0
		\end{align*}
		as $R\to+\infty$. Using $f(r)(f'(r)+1)=O(r)$ at 0, we have
		\begin{align*}
			\int_{(\frac1R,\frac2R)\times N}((n-4)\eta_R^2f^{-1}f''+\eta_R'^2-2\eta_R\eta_R'f^{-1}f')|i_VR^\nabla|^2dV_g\le C\int_{(\frac1R,\frac2R)\times N}|R^\nabla|^2dV_g\to0.
		\end{align*}
		Now we complete the proof of the lemma.
	\end{proof}
	Combining Lemma \ref{S(i_VR) in W-P} and Lemma \ref{energy limit}, we obtain the Theorem \ref{thm 3}
	\begin{theorem}
		If $(n-4)f''<0$ everywhere and $\nabla$ is a weakly stable Yang-Mills connection on $M$ with $R^\nabla\in L^\infty(M)\cap L^2(M)$, then we have
		\begin{align*}
			i_VR^\nabla=0.
		\end{align*}
	\end{theorem}
	\begin{remark}
		Note that the standard sphere with two antipodal points removed  $S^n\backslash\{{(0,...,0,\pm1)}\}=(0,\pi)\times S^{n-1}$ with $g_{S^n}=drdr+\sin^2rg_{S^{n-1}}$ is a warped product manifold, and $f(r)=\sin r$ satisfies the condition in this section. Moreover, we have $f''(r)=-\sin r<0$ in $(0,\pi)$. Since $\frac{\partial}{\partial r}$ may take any direction, we can proof the existence of weakly stable Yang-Mills connections on $S^n$ for $n\ge5$. In fact, the variation constructed in this manner coincides precisely with that selected in \cite{BL}. As a corollary, we have the following proposition.
	\end{remark}
	\begin{corollary}
		Assume that $M$ is an n-dimensional ellipsoid with $n\ge5$. Then there is no weakly stable Yang-Mills connections on $M$.
	\end{corollary}
	
	\section*{Acknowledgement}
	
	The first author is supported by National Key R$\&$D Program of China 2022YFA1005400.

\end{document}